\documentclass{birkjour}
\usepackage{amsmath,amssymb}
\usepackage{xcolor}
\usepackage{frcursive}
\usepackage{comment}
\usepackage{titlesec}
\usepackage[utf8]{inputenc}
\usepackage{amssymb}
\usepackage{enumitem}
\usepackage[hidelinks,bookmarks=true]{hyperref}
\usepackage[numbered]{bookmark}

\makeatletter
\newcommand*{\sectionbookmark}[1][]{%
  \bookmark[%
    level=section,%
    dest=\@currentHref,%
    #1%
  ]%
}
 \newtheorem{thm}{Theorem}[section]
 \newtheorem{cor}[thm]{Corollary}
 \newtheorem{lem}[thm]{Lemma}
 
 \theoremstyle{definition}
 \newtheorem{defn}[thm]{Definition}
 \theoremstyle{remark}
 \newtheorem{rem}[thm]{Remark}

 \numberwithin{equation}{section}

\newcommand{\N}{\mathbb{N}}

\newcommand{\R}{{\rm I\!R}}

\newcommand{\Om}{\Omega}

\newcommand{\ric}{\mathrm{Ric}}

\newcommand{\n}{\mathbf{n}}

\newcommand{\dist}{\mathrm{dist}}

\newcommand{\vol}{\mathrm{Vol}}

\newcommand{\diam}{\mathrm{diam}}

\begin{document}

%-------------------------------------------------------------------------
% editorial commands: to be inserted by the editorial office
%
%\firstpage{1} \volume{228} \Copyrightyear{2004} \DOI{003-0001}
%
%
%\seriesextra{Just an add-on}
%\seriesextraline{This is the Concrete Title of this Book\br H.E. R and S.T.C. W, Eds.}
%
% for journals:
%
%\firstpage{1}
%\issuenumber{1}
%\Volumeandyear{1 (2004)}
%\Copyrightyear{2004}
%\DOI{003-xxxx-y}
%\Signet
%\commby{inhouse}
%\submitted{March 14, 2003}
%\received{March 16, 2000}
%\revised{June 1, 2000}
%\accepted{July 22, 2000}
%
%
%
%---------------------------------------------------------------------------
%Insert here the title, affiliations and abstract:
%
\title[
 Eigenvalues of the Wentzel-Laplace operator]{\begin{center}Isoperimetric bounds for Wentzel-Laplace eigenvalues on Riemannian manifolds
\end{center}}

%----------Author 1
\author[A. Ndiaye]{\begin{center}
A\"issatou M. NDIAYE
\end{center} }

\address{%
 Institut de math\'ematiques\\
			Université de Neuchâtel\\
			   Switzerland \\
		     Tel.: +41327182800\\
     %       http://annuaire.unine.ch/#/person/35565
         }
           
 \email{aissatou.ndiaye@unine.ch}

\thanks{\ldots}
%
%----------classification, keywords, date
%\subjclass{Primary \ldots; Secondary \ldots}

%\keywords{Class file, journal}

\date{\today}
%----------additions
% \dedicatory{To my boss}
%%% ----------------------------------------------------------------------

\begin{abstract}
In this paper, we investigate eigenvalues of the Wentzel-Lapla\\ce
operator on a bounded domain in some Riemannian manifold. We prove
asymptotically optimal estimates, according to the Weyl's law through bounds
that are given in terms of the isoperimetric ratio of the domain. Our
results show that the isoperimetric ratio allows to control the entire
spectrum of the Wentzel-Laplace operator in various ambient spaces.
\end{abstract}

%%% ----------------------------------------------------------------------
\maketitle
%%% ----------------------------------------------------------------------
%\tableofcontents

\section{Introduction}
Let $n\geqslant 2$ and $(M,g)$ be an $n$-dimensional Riemannian manifold. Let $\Omega\subset M$ be a bounded domain with smooth boundary $\Gamma$.
We denote by $\Delta$ and  $\Delta_\Gamma$ the Laplace-Beltrami operators  acting on functions  on $M$ and $\Gamma$, respectively. Notice that, in conformance with conventions
in computational geometry, we define the Laplacian with negative
sign, that is the negative divergence of the gradient operator. The gradient operators on $M$ and $\Gamma$ will be denoted by $\nabla$ and $\nabla_\Gamma$ respectively, the outer normal derivative on $\Gamma$ by $\partial_\n$.  Throughout the paper we denote by $\mathrm{d}_M$ and $\mathrm{d}_\Gamma$ the Riemannian volume elements of $M$ and $\Gamma$.

Given  an arbitrary constant $\beta\in\R$, consider the following eigenvalue problem on $\Omega$:
\begin{equation}\label{w}
\begin{cases}
\Delta u=0 \quad \text{in}~\Omega,\\
\beta \Delta_\Gamma u +\partial_\n u=\lambda u  \quad \text{on}~\Gamma.\\
\end{cases}\quad\text{(Wentzel Problem)}
\end{equation}

 In what follows, we will assume that $\beta$, which we refer to as the boundary parameter, is non-negative. In this case, the Wentzel eigenvalues form a discrete sequence that can be arranged as 
 \begin{equation}\label{spectrum}
 0=\lambda_{W,0}^{\beta}< \lambda_{W,1}^{\beta}\leqslant\lambda_{W,2}^{\beta}\leqslant \cdots\leqslant \lambda_{W,k}^{\beta}\leqslant \cdots \nearrow \infty.
 \end{equation}
We adopt the convention that each eigenvalue is repeated according to its multiplicity. 

The boundary condition in \eqref{spectrum}, which we call Wentzel boundary condition, was  initially introduced in \cite{Wentzel}, in order to find the most general boundary conditions for which the associated operator generates a Markovian semigroup. It is often considered in a more general form  cf.\cite[(1.2)]{Favini2002},\cite[(2.32)]{Gal2015}. Or  sometimes, it subordinates the heat equation  as in \cite[(1.3)]{Kennedy2010} see also \cite{Favini2002}.
A good discussion on motivations and the physical interpretation of Wentzel boundary conditions can be found in \cite{goldstein}. 

The present paper we use valuable tools to find bounds
in terms of geometric quantities in order to estimate Wentzel eigenvalues.
These bounds are optimal according to the asymptotic behaviour of the eigenvalues given by the Weyl law \eqref{WeylW}. 
%%%%%%%%%%%%%%%%%

The eigenvalue problem of the Laplacian with Wentzel boundary
condition has only recently been significantly investigated.
When $\beta=0$, the eigenvalue problem \eqref{w} reduced to the so called Steklov eigenvalue problem.
% There has been much recent work on the Steklov eigenvalue
%problem. See for example\cite{provstub,xiong2017,CGH2018,xia2019escobars}.
% Numerous of these recent results are under new research approach that are of direct interest. 
An advanced reference providing an overwiew on the Steklov problem, is \cite{GP}. As in \cite{Gal2015}, the problem \eqref{w} can be viewed as a perturbed ( as opposite to unperturbed when $\beta=0$) Steklov problem.
%%%%%%%%%%%%%%%%%%
%%Throughout the the rest of this section, $(M, g)$
%% will designate an $n$-dimensional compact Riemannian manifold with boundary $\Gamma$.

The most relevant works on bounds for eigenvalues of the Wentzel-Laplace operator  have been done in
\cite{DKL,Feng1,Feng2}.
Dambrine, Kateb and Lamboley \cite{DKL} obtained a first upper bound for the first non-trivial eigenvalue $\lambda_{W,1}^{\beta}$ in terms of purely geometric quantities if $\Omega$ is a bounded domain in $\R^n$:\\
 Let $\wedge(\Omega)$ denote the spectral radius  of the matrix  $$P(\Omega) \stackrel{\scriptscriptstyle\text{def}}= \left(\int_\Gamma \delta_{ij}-\n_i\n_j  \mathrm{d}_\Gamma\right)_{i,j=1,\ldots,n}.$$ The following inequality holds:
\begin{equation}\label{dkl}
 \lambda_{W,1}^{\beta}\leqslant \frac{\vol(\Omega)+\beta\wedge(\Omega)}{\omega_n^{-\frac{1}{n}}\vol(\Omega)^{\frac{n+1}{n}}\left[1+c_n\left(\frac{\vol(\Omega)\Delta B}{\vol(B)}\right)^2 \right]},\quad c_n:=\frac{(\sqrt[n]{2}-1)(n+1)}{4n}.
\end{equation}
Here, $B$ is the ball having the same volume as $M$  and with the same center of mass than $\Gamma$ and $\omega_{n}$ denotes the volume of the $n$-dimensional Euclidean unit ball. Equality holds in \eqref{dkl} if $M$ is a ball. In \cite{Feng1}, Wang and Xia proved the following bound for the same eigenvalue:
\begin{equation}
 \lambda_{W,1}^{\beta}\leqslant \frac{n\vol(\Omega)+\beta(n-1)\vol(\Gamma)}{n\vol(\Omega)(\vol(\Omega)\omega_n^{-1})^{\frac{1}{n}}}.
\end{equation} 
They also present a bound for $\lambda_{W,1}^{\beta}$ in non-Euclidean case, when the Ricci curvature of $M$ and the principle curvatures of $\Gamma$ are bounded.
Going further, Du-Wang-Xia provided  the following isoperimetric bound for the first $n$($n$ being the dimension) eigenvalues, when 
$M$ is immersed in an Euclidean space $\R^N$ equipped with the canonical Euclidean metric. If $H$ is the mean curvature vector field of $\Gamma$ in $\R^N$ then one has
\begin{equation}\label{dxw}
\frac{1}{n-1}\Sigma_{j=1}^{n-1}\lambda_{W,j}^{\beta}\leqslant\frac{\sqrt{\left[ n\vol(M)+(n-1)\beta\vol(\Gamma)\right]\int_\Gamma  |H|^2 \mathrm{d}_\Gamma}}{\vol(\Gamma)}.
\end{equation}
When $N = n$, that is, $M$ is a bounded domain of $\R^N$, then equality holds in \eqref{dxw} if and only if $M$ is a ball.

%In the same vein
The aim of this work is to go even further and provide  uniformal isoperimetric bounds for all the  eigenvalues of \eqref{w}. If $\Omega$ is a domain of an $n$-dimensional complete Riemannian manifold $(M,g)$, with boundary $\Gamma$, the isoperimetric ratio of $\Omega$ is defined by  $I(\Omega):=\frac{\vol (\Gamma)}{\vol(\Omega)^\frac{n-1}{n}}$. In the numerator $\vol$ stands for the $(n-1)$-Riemannian volume and for the $n$-Riemannian volume from $g$ in the denominator. 

Our first result provides an upper bound in the case of Euclidean domains. We respectively denote $\omega_n$ and $\rho_{n-1}=n\omega_n$  the volumes of the  unit ball and the unit sphere in the $n$-dimensional Euclidean space.
%In our main result, we obtain bounds on \textcolor{red}{\ldots}.
\begin{thm}\label{eucl}
Let $n\geqslant 3$ and $\Omega\subset\R^n$ be a bounded euclidean domain with smooth boundary $\Gamma$. 
Then, for every $k\geqslant 1$, one has
\begin{multline}
\lambda_{W,k}^{\beta}(\Om)
\leqslant  \zeta_1(n)  \left(\frac{\vol(\Omega)}{\vol(\Gamma)}\right)^{1-\frac{2}{n}} \left(\frac{k}{\vol(\Gamma)}\right)^{\frac{2}{n}}\\
+ \zeta_2(n) I(\Omega)^{1+\frac{2}{n-1}} \left[ \frac{\vol(\Omega)}{\vol(\Gamma)}+\beta\right]\left(\frac{k}{\vol(\Gamma)}\right)^\frac{2}{n-1},
\end{multline}
where $\zeta_1(n):= 2^{10(n+1)}\omega_n^{\frac{2}{n}}$ and
 $\zeta_2(n):=\frac{ 2^{10(n+3)}}{n}\omega_n^\frac{1}{n}
$.
\end{thm}
%This leads to the following corollary:
\begin{cor}\label{eucl20052020}
Let $\Omega\subset\R^n$ be a bounded euclidean domain of dimension $n\geqslant 3$ with smooth boundary $\Gamma$. 
Then, for every $k\geqslant 1$, we have
\begin{equation}
\lambda_{W,k}^{\beta}(\Om)
\leqslant C_1(\Omega,\beta)+C_2(\Omega,\beta)\left(\frac{k}{\vol(\Gamma)}\right)^\frac{2}{n-1}.
\end{equation}
Here $ C_1(\Omega,\beta)$ and $ C_2(\Omega,\beta)$ are geometric constants given by :
%\begin{equation*}
%\begin{cases} 
%C_2(\Omega,\beta)=1+ \zeta_2(n) I(\Omega)^{1+\frac{2}{n-1}} \left[ \frac{\vol(\Omega)}{\vol(\Gamma)}+\beta\right]\\
%\text{and }\\
% C_1(\Omega,\beta)=\zeta_1^n(n) \left(\frac{\vol(\Omega)}{\vol(\Gamma)}\right)^{n-2}C_2(\Omega,\beta).
%\end{cases}
%\end{equation*}
\begin{align*}
 C_2(\Omega,\beta)&= \zeta_2(n) I(\Omega)^{1+\frac{2}{n-1}} \left[ \frac{\vol(\Omega)}{\vol(\Gamma)}+\beta\right]+1 \\
C_1(\Omega,\beta)&=\zeta_1^n(n) \left(\frac{\vol(\Omega)}{\vol(\Gamma)}\right)^{n-2}C_2(\Omega,\beta).
\end{align*}
 The constants $\zeta_1(n)$ and
 $\zeta_2(n)$ are the same as in Theorem \ref{eucl}.
\end{cor}
For bounded domains in Riemannian manifold with Ricci curvature bounded from below, we have an isoperimetric upper bound, which also depends on the infimum isoperimetric ratio that we define as follows:
\begin{defn}
Let $(M,g)$ be a complete Riemannian manifold of dimension $n\leqslant 2$ and $\Omega$ a bounded domain in $M$. The infimum isoperimetric  ratio of $\Omega$ is the quantity $
I_0(\Omega):=\inf\{I(U) : U\text{ open set in }\Omega\}
$. In particular, if $\Om$ is an  Euclidean domain, one has $I_0(\Omega)=I_0(\R^n)=n\omega_n^\frac{1}{n}$. 
\end{defn}

\begin{thm}\label{t16011}
Let $(M,g)$ be a complete Riemannian manifold of dimension $n\geqslant 3$ with non-negative Ricci curvature. Let 
$\Omega\subset M$ be a bounded domain with smooth boundary $\Gamma$. Then for every $k\geqslant 1$, we have
\begin{multline}
\lambda_{W,k}^{\beta}(\Om)
\leqslant  c_1(n)  \left(\frac{\vol(\Omega)}{\vol(\Gamma)}\right)^{1-\frac{2}{n}} \left(\frac{k}{\vol(\Gamma)}\right)^{\frac{2}{n}}\\
+ c_2(n) \left(\frac{I(\Omega)}{I_0(\Omega)} \right)^{1+\frac{2}{n-1}} \left[ \frac{\vol(\Omega)}{\vol(\Gamma)}+\beta\right]\left(\frac{k}{\vol(\Gamma)}\right)^\frac{2}{n-1},
\end{multline}
where $c_1(n):= 2^{10(n+1)}\omega_n^{\frac{2}{n}}$ and
 $c_2(n):= 2^{5(n+5)}\rho_{n-1}^{\frac{2}{n-1}}$.
\end{thm}
\begin{rem}
It is not usually simple to gauge this quantity $I_0(\Omega)$. It is not easy to determinate the best constant in the isoperimetric inequality for domains in many  complete Riemmannian manifolds. For example, as we see in Corollary \ref{cor1701}, the longstanding conjecture known as the Cartan-Hadamard conjecture, is about sharp isoperimetric inequalities in complete Riemannian manifolds with negative sectional curvature.
\end{rem}
\begin{thm}\label{t16012}
Let $(M,g)$ be a complete Riemannian manifold of dimension $n\geqslant 3$ with Ricci curvature bounded from below by $-(n-1)\kappa^2$, $\kappa\in\R_{> 0}$. Let 
$\Omega\subset M$ be a bounded domain with smooth boundary $\Gamma$. Then for every $k\geqslant 1$, we have
\begin{equation}\label{ineq22052020}
\lambda_{W,k}^{\beta}(\Om)
\leqslant 
A(\Omega,\beta)+B(\Omega,\beta)\left(\frac{k}{\vol(\tilde{\Gamma)}}\right)^\frac{2}{n-1},
\end{equation}
where 
\begin{align*}
A(\Omega,\beta) &=
\kappa^2  \zeta(n)\left\{ 1+ \left(\kappa\frac{\vol(\Omega)}{\vol(\Gamma)}\right)^{1-\frac{2}{n}}+\left(\frac{I(\Omega)}{I_0(\Omega)} \right)^{1+\frac{2}{n-1}} \left[ \kappa\frac{\vol(\Omega)}{\vol(\Gamma)}+\beta\right]\right\},\\
B(\Omega,\beta) &=\zeta(n)\Bigg\{1+\left(\frac{I(\Omega)}{I_0(\Omega)} \right)^{1+\frac{2}{n-1}} \left[ \kappa\frac{\vol(\Omega)}{\vol(\Gamma)}+\beta\right]\Bigg\}
,
\end{align*}
$\zeta(n)$ being a constant depending only on the  dimension $n$.
\end{thm}
Theorems \ref{t16011} and \ref{t16012} emanate from a generic result (Theorem \ref{t1501}) that we prove in Section \ref{section1701}.

Besides the Euclidean case, we have at least one other situation where we know something about the quantity $I_0(\Omega)$.
The so called Cartan-Hadamard conjecture, proved in dimensions $n = 2$ by Weil \cite{Weil1926}, $n = 3$ by Kleiner \cite{Kleiner1992} and $n = 4$ by Croke \cite{Croke1984}, states that any bounded domain in a smooth Cartan-Hadamard manifold of dimension $n$ satisfies
$$I(\Omega)\geqslant C(n)$$
where $C(n)$ is a dimensional constant.
 Very recently,  Ghomi  and  Spruck (2019) in \cite{GhomiSpruck2019} proposed a solution in all dimensions. This leads to the following corollary.
 \begin{cor}\label{cor1701}
Let $(M,g)$ be a smooth Cartan-Hadamard manifold of dimension $n\geqslant 3$  with Ricci curvature bounded from below by $-(n-1)\kappa^2$, $\kappa\in\R_{>0}$ and $\Omega\subset M$ a bounded domain with smooth boundary $\Gamma$. Then for every $k\geqslant 1$, we have
\begin{equation}
\lambda_{W,k}^{\beta}(\Om)
\leqslant 
A(\Omega,\beta)+B(\Omega,\beta)\left(\frac{k}{\vol(\tilde{\Gamma)}}\right)^\frac{2}{n-1},
\end{equation}
where 
\begin{align*}
A(\Omega,\beta) &=
\kappa^2  \zeta(n)\left\{ 1+ \left(\kappa\frac{\vol(\Omega)}{\vol(\Gamma)}\right)^{1-\frac{2}{n}}+I(\Omega)^{1+\frac{2}{n-1}} \left[ \kappa\frac{\vol(\Omega)}{\vol(\Gamma)}+\beta\right]\right\},\\
B(\Omega,\beta) &=\zeta(n)\Bigg\{1+I(\Omega)^{1+\frac{2}{n-1}} \left[ \kappa\frac{\vol(\Omega)}{\vol(\Gamma)}+\beta\right]\Bigg\}
,
\end{align*}
$\zeta(n)$ being a constant depending on the  dimension $n$.
\end{cor}

%This result shows that it is possible to bound each eigenvalue $\lambda_{W,k}^{\beta}$ by simply controlling the
%volume of $M$.
%The main technical contribution of our paper is

%\paragraph{Methods}
%\paragraph{\textbf{Plan of the paper}}
%\paragraph{Comment}
% One purpose of this paper is to prove ... theorem with the same tools\\
%  Use the same technique as in\\
%   as they share the same technical structure, and as they use the same tool to be solved.\\
%   based on the comparison geometry of principal curvatures of hypersurfaces that are parallel to
%the boundary.\\

\paragraph*{Plan of the paper.} The proofs of Theorems \ref{eucl} \ref{t16011} and \ref{t16012} are presented in Section \ref{sec1701}. We present the proof of Theorem \ref{t1501} in Section 3, following some technical results. We devote Section 2  to briefly summarize properties of the Wentzel-Laplace eigenvalues.

% ----------------------------------------------------------------------

 \section{Wentzel-Laplace operator and functional framework}
 Consider the map  $ \wedge : L^2 (\Gamma) \longrightarrow L^2 (\Omega) $  related to the  Dirichlet problem
\begin{equation}\label{harm}
\begin{cases}
\Delta u=0\quad\text{in } \Omega,\\
u |_\Gamma= f \quad \text{on } \Gamma,
\end{cases}
\end{equation}
which associate to any $ f \in  L^2 (\Gamma)$ its harmonic extension  i.e. the unique function $u$ in $ L^2(\Omega)$ satisfying \eqref{harm}. This map is well defined from $L^2(\Gamma)$ (respectively, $H^{\frac{1}{2}}(\Gamma)$) to $L^2(\Omega)$ (respectively, $H^1(\Omega)$). See \cite[p. 320, Prop 1.7]{Taylor}, for more details. By $H^s(\Omega)$ and $H^s(\Gamma)$, we denote the Sobolev spaces of order $s$ on $\Omega$ and $\Gamma$, respectively, and $u |_\Gamma\in H^\frac{1}{2}(\Gamma)$ stands for the trace of $u\in H^1(\Omega)$ at the boundary $\Gamma$. This will also be denoted simply by $u$, if no ambiguity can result.

Then the Dirichlet-to-Neumann operator is defined by 
\begin{align}
\mathrm{N_ D}:&  H^\frac{1}{2}(\Gamma) \mapsto  H^{-\frac{1}{2}} (\Gamma)\\
 			&f\mapsto \partial_\n (\wedge f)\nonumber.
\end{align}
Again $\partial_\n$  stands for the  normal derivative at the boundary $\Gamma$ of $\Omega$ with unit normal vector $\n$ pointing outwards.

For all $u\in H^{\frac{1}{2}}(\Gamma)$, we define the operators $\mathrm{B}_0= N_D$ (in the operator sense). For  $\beta>0$, we define for all $u\in H^1(\Gamma)$ $\mathrm{C}_\beta u \stackrel{\scriptscriptstyle\text{def}}= \beta\Delta_\Gamma u$ and $\mathrm{B}_\beta \stackrel{\scriptscriptstyle\text{def}}= \mathrm{B}_0+\mathrm{C}_\beta$.
The eigenvalue sequence $\{ \lambda_{W,k}^{\beta}\}_{k=0}^\infty $ given in \eqref{spectrum} can be interpreted as the spectrum associated with the operator $\mathrm{B}_\beta$ and  is subject to the following min-max characterization (see e.g., \cite[Thm1.2]{sandgren} and \cite[(2.33)]{Gal2015}):
\paragraph*{Min-max principle.}
Let $\mathfrak{V}(k) $ denotes  the set of  all $ k$-dimensional  subspaces of   $ \mathfrak{V}_\beta$ which  is defined by
\begin{align*}
 \mathfrak{V}_0&\stackrel{\scriptscriptstyle\text{def}}=\{(u,u_\Gamma)\in H^1(\Omega)\times H^\frac{1}{2}(\Gamma): u_\Gamma=u|_\Gamma \},\\
 \mathfrak{V}_\beta&\stackrel{\scriptscriptstyle\text{def}}=\{(u,u_\Gamma)\in H^1(\Omega)\times H^1(\Gamma): u_\Gamma=u|_\Gamma \}, \quad \beta>0.
\end{align*}
Of course, for all $\beta>0$, we have $ \mathfrak{V}_\beta\subset \mathfrak{V}_0$. For every $k\in\N$, the $k$Th eigenvalue of the Wentzel-Laplace operator $B_\beta$ satisfies

\begin{equation}\label{char}
 \lambda_{W,k}^{\beta}={\underset{V\in \mathfrak{V}(k)}{\min}  }\underset {0\neq u\in V} {\max} R_\beta(u), \quad k\geqslant 0,
\end{equation}
where $R_\beta(u) $,  the Rayleigh quotient for $\mathrm{B}_\beta$, is given by 
\begin{equation}\label{rayleigh}
R_\beta(u) \stackrel{\scriptscriptstyle\text{def}}=\frac{\int_\Om{|\nabla u|^2 \mathrm{d}_M+\beta\int_{\Gamma}{|\nabla_\Gamma u|^2 \mathrm{d}_\Gamma}}}{\int_{\Gamma}{u^2 \mathrm{d}_\Gamma}}, \quad \text{for all } u\in \mathfrak{V}_\beta\backslash\{0\}.
\end{equation}

\paragraph{Asymptotic behaviour.}
The eigenvalues for the Dirichlet-to-Neumann map $B_0=N_ D$ are those of the well known Steklov problem. 
\begin{equation}\label{Steklov}
\begin{cases}
\Delta u=0, & {\rm dans\ }\Omega,\\
\partial_\n u=\lambda^S u, & {\rm sur\ }\Gamma.
\end{cases}
\end{equation}
A good discussion of this problem can be found in \cite{GP}. The Steklov eigenvalues are then  $ \{ \lambda_{W,k}^{0}\}_{k=0}^\infty$
which we shall denote equivalently as $\{ \lambda_{S,k}\}_{k=0}^\infty $.
 They behave according to the following asymptotic formula:
 \begin{equation}
\label{WeylS}
\lambda_{S,k}= C_n k^{\frac{1}{n-1}}+O(k^{\frac{1}{n-1}}),\quad k\rightarrow\infty.
\end{equation}
where  $C_n=\frac{2\pi}{(\omega_{n-1} \vol(\Gamma)^{\frac{1}{n-1}}} $. The reader can refer to \cite[section 4]{sandgren}. For $\beta>0$, the Weyl asymptotic for $ \lambda_{W,k}^{\beta}$  can be deduced  directly from properties of perturbed forms using the   asymptotic behaviour of the spectrum of $C_\beta$ by Hörmander:
\begin{equation}
\lambda_{C_{\beta,k}}=\beta C_n^2k^{\frac{2}{n-1}}+O(k^{\frac{2}{n-1}}),\quad k\rightarrow\infty.
\end{equation}
The Weyl law for eigenvalues on problem \eqref{w} reads
 \begin{equation}
\label{WeylW}
\lambda_{W,k}^{\beta}=\beta C_n^2k^{\frac{2}{n-1}}+O(k^{\frac{2}{n-1}}),\quad k\rightarrow\infty.
\end{equation}
A  complete and detailed discussion  about the spectral properties of the  Wentzel Laplacian can be found in \cite[Section 2]{Gal2015} and for the asymptotic in \eqref{WeylW}, the reader can refer to \cite[Prop 2.7 and  (2.37)]{Gal2015}.
 
\section{General inequality}\label{section1701}
In this section, we establish some needed technical results and the major result in this paper used to prove our main theorems. 
Let $n\geqslant 2$ and $(M,g)$ be an $n$-dimensional Riemannian manifold. Let $\Omega\subset M$ a bounded domain with smooth boundary $\Gamma$.
 Let $r\in\R_{>0}$, we denote by $B(x,r)=\{p\in M, d(x,p)<r\}$ the metric ball of radius $r$ centered at $x\in M$, where $d$ is  the Riemannian distance associated to the metric $g$. 
We assume $\Gamma$ satisfies the following hypothesis:
\begin{enumerate}
\item[$(H_0)$:]\label{H0}
 There exists a radius $r_-(\Gamma)>0$ and a constant $C\in\N_{> 1}$ such that for all $x\in\Gamma$ and $r<r_-(\Gamma)$, one has 
\begin{equation} \label{Vboules}
  \vol(B(x,r))<C \omega_nr^n\quad\text{and}\quad \vol(\partial B(x,r))<C \rho_{n-1}r^{n-1}.
\end{equation} 
\end{enumerate}
Here $\partial B(x,r)$ denotes the geodesic sphere of radius $r$ centered at $x$.

 \begin{lem}\label{entrop}Let $(M,g)$, $\Omega$ and $\Gamma$ be as above. 
For every $K\in\N$, let  $r_K$  be an associated ``maximal" radius defined by
\begin{equation}
r_K:=\left(\frac{\vol(\Omega)}{2}\right)^\frac{1}{n} \left(\frac{I_0(\Omega)}{K C\rho_{n-1}}\right)^\frac{1}{n-1}.
\end{equation} 
Let $\{x_j\}_{j=1}^K$ be an arbitrary set of points  in $\Gamma$.
Then for every  $r>0$ satisfying both $r< \frac{1}{2}r_-(\Omega)$ and $r\leqslant \frac{1}{2} r_K$, one has
\begin{equation}
\vol\left(\Gamma\backslash \bigcup_{j=1}^KB(x_j,2r)\right)>0.
\end{equation}
 \end{lem}

 \begin{proof}
 %Take an arbitrary set of points $\{x_1,\dots,x_K\}$ in $\Gamma$. 
 We denote by $\Omega_0$ (respectively  $\Gamma_0$) the subset $\Omega\backslash \bigcup_{j=1}^KB(x_j,2r)$ (respectively $\Gamma\backslash \bigcup_{j=1}^KB(x_j,2r)$). One can think of $\overline{\Omega}_0$ as a holed cheese. 
Since the boundary of $\Omega_0$, that we denote by $\partial\Omega_0$, is contained in the union $\Gamma_0\bigcup \left(\bigcup_{j=1}^K\vol(\partial B(x_j, 2r))\right)$, one has
\begin{align*}\nonumber
 \vol(\Gamma_0)&\geqslant \vol(\partial\Omega_0)- \sum_{j=1}^K\vol(\partial B(x_j, 2r))\\
 &= I(\Omega_0)\vol(\Omega_0)^{\frac{n-1}{n}}- \sum_{j=1}^K\vol(\partial B(x_j,2r)),
\end{align*}
where $ \partial B(x_j,2r)=\{p\in M\mid d(x_j,p)=2r\} $.\\
Then, since $2r<r_-(\Gamma)$, one has 
\begin{equation}
 \vol(\Gamma_0)+K C\rho_{n-1}(2r)^{n-1}> I(\Omega_0)\vol(\Omega_0)^{\frac{n-1}{n}}\label{eqd1}.
\end{equation}
Now, we assume that 
 \begin{equation}\label{eqd2}
I_0(\Omega)^\frac{n}{n-1}\vol(\Omega)- K C\rho_{n-1}^\frac{n}{n-1}(2r)^n\geqslant 0.
  \end{equation}
Noticing that $ I_0(\Omega)\leqslant I_0(\R^n)$, we have then
\begin{align*}
 I(\Omega_0)\vol(\Omega_0)^{\frac{n-1}{n}} &> I(\Omega_0)[\vol(\Omega)- K C\omega_n(2r)^n]^\frac{n-1}{n}\\ 
  &\geqslant [I_0(\Omega)^\frac{n}{n-1}\vol(\Omega)- K C\rho_{n-1}^\frac{n}{n-1}(2r)^n]^\frac{n-1}{n}.\\
 \end{align*}
 Replacing in \eqref{eqd1}, this leads to the following inequality:
 \begin{equation}\label{eqd0}
 \vol(\Gamma_0)>[I_0(\Omega)^\frac{n}{n-1}\vol(\Omega)- K C\rho_n^\frac{n}{n-1}(2r)^n]^\frac{n-1}{n}-K C\rho_{n-1}(2r)^{n-1}.
 \end{equation}
The right hand side  is non-negative if 
\begin{equation}\label{eqanniv}
\left(\frac{I_0(\Omega)}{K C\rho_{n-1}}\right)^\frac{n}{n-1}\frac{\vol(\Omega)}{ (K C)^{-\frac{1}{n-1}}+1}\geqslant (2r)^{n}.
\end{equation}
We Notice that $\frac{1}{ (K C)^{-\frac{1}{n-1}}+1}\geqslant \frac{1}{2}$. Inequality \eqref{eqanniv} is then satisfied whenever
 \begin{equation*}
 r \leqslant\frac{1}{2} \left(\frac{I_0(\Omega)}{K C\rho_{n-1}}\right)^\frac{1}{n-1}\left(\frac{\vol\Omega)}{2}\right)^\frac{1}{n}=\frac{1}{2}r_K.
 \end{equation*}
This implies that $\vol(\Gamma_0)>0$, under the assumption in \eqref{eqd2}. 
 However, \eqref{eqd2} can be written as  
 \begin{equation}\label{eq1601}
 r\leqslant\frac{1}{2}\left(\frac{\vol(\Omega)}{KC}\right)^\frac{1}{n}\left( \frac{I_0(\Omega)}{\rho_{n-1}}\right)^\frac{1}{n-1}=\frac{2^{\frac{1}{n}}}{2}(KC)^{\frac{1}{n(n-1)}}r_K. 
 \end{equation}
Since $1\leqslant 2^{\frac{1}{n}}(KC)^{\frac{1}{n(n-1)}}$, \eqref{eq1601} is satisfied by assumption. This ends the proof.
 \qedhere
 \end{proof}

 \begin{lem}\label{entrop1}
  Let the assumptions of Lemma \ref{entrop} be fulfilled.
 We define 
\begin{equation}\label{eqd4}
K_0:=\Biggl\lfloor {\frac{I_0(\Omega)}{C\rho_{n-1}r_-(\Gamma)^{n-1}} \left(\frac{\vol(\Omega)}{2}\right)^{\frac{n-1}{n}}}\Biggr\rfloor+1,
\end{equation} 
 where $\lfloor~\rfloor$ denotes the floor function, so that 
 $r_K <  r_-(\Gamma)$ if $K\geqslant K_0$.
 Let $\{x_j\}_{j=1}^K$ be an arbitrary set of points in $\Gamma$. Then, for every $K\geqslant K_0$ and 
$0<r\leqslant \frac{1}{16} r_k$, we have
\begin{equation}
\vol\left(\Gamma\backslash \bigcup_{j=1}^KB(x_j,2r)\right)> \left(\frac{r}{r_K}\right)^{\frac{n-1}{n}}I_0(\Omega)\vol(\Om)^{\frac{n-1}{n}}.
\end{equation}
\end{lem}

%%
%\begin{lem}
%Under the hypotheses of Lemma \ref{entrop}, if $0<r\leqslant\frac{1}{4} r_K$ and $K\geqslant K_0$, then
%
%for every arbitrary set of points $\{x_j\}_{j=1}^K$ in $\Gamma$.
%\end{lem}
\begin{proof}
From \eqref{eqd0} in the proof of Lemma \ref{entrop}, one  has
$$ \vol(\Gamma_0)>[I_0(\Omega)^\frac{n}{n-1}\vol(\Omega)- K C\rho_n^\frac{n}{n-1}(2r)^n]^\frac{n-1}{n}-K C\rho_{n-1}(2r)^{n-1}.$$
Setting $\alpha:=\frac{r_K}{r}$ (we notice that $\alpha\geqslant 2^4$ since $r\leqslant\frac{1}{2^4}r_K$), we have
\begin{align}
(2r)^n & = \left(\frac{2}{\alpha}r_K\right)^{n}=\left(\frac{2}{\alpha}\right)^{n}\frac{\vol(\Om)}{2} \left(\frac{I_0(\Omega)}{KC\rho_{n-1}}\right)^{\frac{n}{n-1}}\nonumber\\
&\leqslant\frac{1}{KC\rho_{n-1}^{\frac{n}{n-1}}}\left(\frac{2^{\frac{n-1}{n}}}{\alpha}\right)^{n} I_0(\Omega)^{\frac{n}{n-1}}\vol(\Om),\label{eqf1}
\end{align}
where we have used that $KC\geqslant 1$.
On the other hand,
\begin{align}
(2r)^{n-1} & = \left(\frac{2}{\alpha}r_K\right)^{n-1}=\left(\frac{2}{\alpha}\right)^{n-1}\left(\frac{\vol(\Om)}{2}\right)^{\frac{n-1}{n}} \frac{I_0(\Omega)}{KC\rho_{n-1}}\nonumber\\
&\leqslant\frac{1}{KC\rho_{n-1}}\left(\frac{2^{\frac{n-1}{n}}}{\alpha}\right)^{n-1}I_0(\Omega)^{\frac{n}{n-1}}\vol(\Om).\label{eqf2}
\end{align}
From inequalities \eqref{eqf1} and \eqref{eqf2}, we get
\begin{multline}
\vol\left(\Gamma\backslash \bigcup_{j=1}^KB(x_j,2r)\right)\\
> \left[\left(1-\left(\frac{2^{\frac{n-1}{n}}}{\alpha}\right)^{n}\right)^{\frac{n-1}{n}}-\left(\frac{2^{\frac{n-1}{n}}}{\alpha}\right)^{n-1} \right]I_0(\Omega)\vol(\Omega)^{\frac{n-1}{n}}.
\end{multline}
We notice that, since $\alpha>2$,
\begin{align*}
\left(1-\left(\frac{2^{\frac{n-1}{n}}}{\alpha}\right)^{n}\right)^{\frac{n-1}{n}}-\left(\frac{2^{\frac{n-1}{n}}}{\alpha}\right)^{n-1} &\geqslant \left(1-\alpha^{-1} \right)^{\frac{n-1}{n}}-\alpha^{-\frac{n-1}{n}}\\
 &=\alpha^{-\frac{n-1}{n}}\left[\left(\alpha-1\right)^{\frac{n-1}{n}}-1\right]\\
  &\geqslant\alpha^{-\frac{n-1}{n}}\left[15^{\frac{n-1}{n}}-1\right].
\end{align*}

It follows that $$\vol\left(\Gamma\backslash \bigcup_{j=1}^KB(x_j,2r)\right)> \alpha^{-\frac{n-1}{n}} I_0(\Omega)\vol(\Omega)^{\frac{n-1}{n}},$$
 since $15^{\frac{n-1}{n}}\geqslant 2$ for every $n\geqslant 2$.
\end{proof}
Let $(M,g)$, $\Omega$ and $\Gamma$ be as described above and $r\in\R_{>0}$. 
The external covering number ${N^{ext}_r(\Gamma)}$ of $\Gamma$ in $M$ with respect to $r$ is defined as the fewest number of points $x_1,\dots,x_N \in M$ such that the balls $B(x_1,r),\dots,B(x_N,r)$ cover $\Gamma$. Lemmas \ref{entrop} and \ref{entrop1} imply the following principal lemma.
\begin{lem}\label{lemla}
Let $n\geqslant 2$ and $(M,g)$ be an $n$-dimensional Riemannian manifold. Let $\Omega\subset M$ a bounded domain with smooth boundary $\Gamma$. Then for every $K\geqslant K_0$ and $ 0<r\leqslant \frac{1}{2} r_K$,
 \begin{enumerate}[label=\roman*.]
 \item \label{eqla1}
$K< N^{ext}_r(\Gamma).$
\item \label{eqla2} If in addition $r\leqslant \frac{1}{16} r_K$ then for every arbitrary set of points $\{x_j\}_{j=1}^K$ in $M$, one has
\begin{equation}
\vol(\Gamma\backslash \bigcup_{j=1}^KB(x_j,r))> \left(\frac{r}{r_K}\right)^{\frac{n-1}{n}}I_0(\Omega)\vol(\Om)^{\frac{n-1}{n}}.
\end{equation}
 \end{enumerate}
\end{lem}
\begin{proof}
Suppose ${N^{ext}_r(\Gamma)}\leqslant K$ and let $\{B(x_j,r)\}_{j=1}^{N^{ext}_r(\Gamma)}$ be a minimal covering of $\Gamma$. By the minimality assumption, every $B(x_j,r)$ intersects $\Gamma$. For $j\in\{1,\ldots,{N^{ext}_r(\Gamma)}\}$, let $x'_j\in B(x_j,r)\cap \Gamma$, one has 
\begin{equation*}
B(x_j, r)\subset B(x'_j,2r),\quad \text{ for every } i\in\{1,\ldots, {N^{ext}_r(\Gamma)}\}.
\end{equation*}
This implies $$\vol\left(\Gamma\backslash \bigcup_{j=1}^{N^{ext}_r(\Gamma)}B(x'_j,2r)\right)\leqslant \vol\left(\Gamma\backslash \bigcup_{j=1}^{N^{ext}_r(\Gamma)}B(x_j,r)\right).$$
We complete the family  $\{B(x'_j,2r)\}_{j=1}^{N^{ext}_r(\Gamma)}$ to  $\{B(x'_j,2r)\}_{j=1}^K$ by setting $x'_j:=x'_1$ for ${N^{ext}_r(\Gamma)}<j\leqslant K$.
Then, applying Lemma \ref{entrop}, we have 
\begin{align*}
\vol(\Gamma\backslash \bigcup_{j=1}^{N^{ext}_r(\Gamma)}B(x_j,r))&\geqslant \vol(\Gamma\backslash \bigcup_{j=1}^{N^{ext}_r(\Gamma)}B(x'_j,2r))\\
&=\vol(\Gamma\backslash \bigcup_{j=1}^K B(x'_j,2r))>0.
\end{align*}
Hence, it is contradictory to $\Gamma\subset \bigcup_{j=1}^{N^{ext}_r(\Gamma)}B(x_j,r))$.
To prove \eqref{eqla2}, we notice that if $B(x_j,r)\cap \Gamma\neq \emptyset$ then $B(x_j,r)\subset B(x'_j,2r)$ with $x'\in \Gamma$. The inequality follows the applying Lemma \ref{entrop1}.
\end{proof}
%
%We will use, to proof $\ldots$, the following result which is a global version of  Lemme$\ldots$:

The next lemma of Colbois and Maerten \cite{CM} provides the final ingredient to prove the most technical results in this paper presented in Theorems \ref{mainanniv} and \ref{t1501}.

Let $(X, d)$ be a complete, locally compact metric space. Let $\varepsilon\in\N$ and $N:(0,\rho]\longrightarrow\N_{\geqslant 2}$ an  increasing function.  We say that 
$(X, d)$ satisfies the $(N,\varepsilon)$-covering property if each ball of radius $r$ can be covered by $N(r)$ balls of radius $\frac{r}{\varepsilon}$.  In order to simplify notation, we will write $N_r$ instead of  $N(r)$.

We denominate capacitor any  couple $(A, B)$ of subsets such  that $\emptyset\neq A \subset B \subset X$. Two capacitors $(A_1, B_1)$ and $(A_2, B_2)$ are disjoint if $B_1\cap B_2=\emptyset$.  A family of capacitors is a finite set of capacitors in $X$ that are pairwise disjoint. 

\begin{lem}[Colbois-Maerten, $2008$]\label{CMglobal}
Let  $(X, d, \mu)$ be a complete, locally compact metric measure space satisfying the $(N,4)$-covering property with $N:(0,\diam(X)]\longrightarrow \N_{\geqslant 2}$ a discrete positive function.
Let $r>0$ and $K\in\N$ such that  for every $x\in X$, $\mu(B(x,r))\leqslant\frac{\mu(X)}{4KN_r^2}$. Then there exists a family of $K$ capacitors $\{(A_i,B_i)\}_{1\leqslant i\leqslant K}$ with the following properties for $1\leqslant i,j\leqslant K$
\begin{enumerate}
\item $\mu(A_i)\geqslant \frac{\mu(X)}{2N_rK}$,
 \item $B_i=A_i^r:= \{x \in X,~ d(x,A_i) < r\}$ is
the $r$-neighbourhood of $A_i$ and  $d(B_i,B_j)>2r$ whenever $i\neq j$.
\end{enumerate} 
\end{lem}
%For the sake of completeness we will present the proof of Lemma \ref{CMglobal} in the Appendix.
\begin{thm}\label{mainanniv}
Let $(M,g)$ be a complete Riemannian manifold of dimension $n\geqslant 2$. Let 
$\Omega\subset M$ be a bounded domain whose boundary $\Gamma$ is a
smooth hypersurface satisfying $(H_0)$. 
We assume that $M$, with respect to the distance associated to the metric $g$ satisfies the $(N,4)$-covering property for some discrete positive function $N$. \\
Then, for every integer $k\geqslant \frac{1}{4} K_0$ ( $K_0$ is the same as in \eqref{eqd4}), one has
\begin{multline}
\lambda_{W,k}^{\beta}(\Om)
\leqslant 
 C_1  \left(\frac{\vol(\Omega)}{\vol(\Gamma)}\right)^{1-\frac{2}{n}} \left(\frac{k}{\vol(\Gamma)}\right)^{\frac{2}{n}}\\
+ C_2  \left(\frac{I(\Omega)}{I_0(\Omega)} \right)^{1+\frac{2}{n-1}} \left[\frac{\vol(\Omega)}{\vol(\Gamma)}+\beta\right]\left(\frac{k}{\vol(\Gamma)}\right)^\frac{2}{n-1},
\end{multline}
where $C_1= 2^8(C\omega_n)^{\frac{2}{n}} N_r^2$, $C_2=2^{21} (C\rho_{n-1})^\frac{2}{n-1}N_r$ and $r:=\frac{1}{64}r_{(4k)}$.
\end{thm}

\begin{proof}%[Proof of Theorem \ref{mainanniv}]
The methods we use in this proof are inspired by \cite{CEG2013}. 
We consider the  metric measure space $(M,d, \mu)$, where  $d$ is the distance from the metric $g$ and $\mu$ is the Borel measure with support $\Gamma$ defined for each Borelian $A$ of  $M$ by
$$\mu(A):=\int_{A\cap\Gamma} \mathrm{d}_{\Gamma}.$$
Fix  $K=4k$ and  choose in  $M$ a family of points $\{x_j\}_{j=1}^{K}$ satisfying 
\begin{equation}
\begin{cases}
B(x_j,2r)\cap B(x_i,2r)=\emptyset\qquad \text{ for all } 1\leqslant i\neq j\leqslant K,\\
\mu(B(x_1,r))\geqslant\mu(B(x_2, r))\geqslant\ldots\geqslant\mu(B(x_{K},r))\geqslant\mu(B(x,r)),\label{eqg}
\end{cases} 
\end{equation}
for all $x\in M_0:=M\backslash \bigcup_{j=1}^{K}B(x_j,4r) $.
This can be done inductively, selecting the point $x_1$ such that 

$$\mu (B(x_1,r))=\sup\{\mu (B(x,r)),~  x\in M\},$$
and the points  $x_j$, for $j=2,\ldots, K$, such that 
$$\mu (B(x_j,r))=\sup\{\mu (B(x,r)),~  x\in M\backslash \bigcup_{i=1}^{j-1}B(x_i,4r)\}.$$
% $B(a_j,\frac{1}{2} r)\cap \Gamma\neq \emptyset$,\label{eqa2}
%\item $\vol(\Gamma\backslash \bigcup_{j=1}^{4k}B(x_j,2r))> \frac{1}{4^n}I_0(\Omega)\vol(\Om)^{\frac{n-1}{n}}$.\label{eqa3}
%
%for $x\in M_0:=M\backslash \bigcup_{j=1}^{4k}B(x_j,4r) $ where $\mu(A):=\vol(A\cap\Gamma)$ for every borelian $A$ in $M$.
There are two possible cases: 
\paragraph*{Case $\mu(B(x_{K},r))\leqslant\frac{\mu(M)}{4KN_r^2}$.}
We consider the metric measure space\\ $(M_0, d, \mu_0)$ where $\mu_0$ is defined by 
$$\mu_0(A):=\int_{A\cap \Gamma_0}   \mathrm{d}_{\Gamma},\quad \Gamma_0:=\Gamma\backslash \bigcup_{i=1}^{K}B(x_i,4r).$$
 for every Borelian A in $M$.
Since $4r=\frac{1}{16}r_{K}$, it follows from Lemma \ref{lemla}, that 
 \begin{equation*}
 \mu_0(M_0)=\vol(\Gamma_0)>  \frac{1}{16^{\frac{n-1}{n}}}I_0(\Omega)\vol(\Om)^{\frac{n-1}{n}}.
 \end{equation*}
  From \eqref{eqg} one has
$\mu_0(B(x,r))\leqslant\mu(B(x,r))\leqslant\frac{\mu(M)}{4KN_r^2}$
for every $x\in M_0$. Applying Lemma \ref{CMglobal}, we have a family of $K$ capacitors $\{(A_i,B_i)\}_{1\leqslant i\leqslant K}$ with the following properties for $1\leqslant i,j\leqslant K$:
\begin{enumerate}
\item $\mu_0(A_i)\geqslant \frac{\mu_0(M_0)}{2N_rK}$,
 \item $B_i=A_i^r= \{x \in X ,~ d(x,A_i) < r\}$ is
the $r$-neighborhood of $A_i$ and  $d(B_i,B_j)>2r$ whenever $i\neq j$.
\end{enumerate} 
We notice that $\mu_0(M_0)=\vol(\Gamma_0)$.\\
%%%%
For each $1\leqslant j \leqslant K$, we consider the function $\varphi_j$ supported
in $A_j^{r}$ defined by 
$$\varphi_j(x):=
\begin{cases}
1-\frac{d(A_j,x)}{r} \quad &\forall~ x\in A_j^{r},\\
0\quad &\forall~ x\in M\backslash A_j^{r}.
\end{cases}
$$
It follows that $R_\beta(\varphi_j)\leqslant\frac{\int_{\Om\cap A_j^{r}}|\nabla \varphi_j|^2 \mathrm{d}_M+\beta\int_{\Gamma\cap A_j^{r}}{|\nabla {\varphi_j}|^2 \mathrm{d}_\Gamma}}{\int_{\Gamma\cap A_j}{{\varphi_j}^2 \mathrm{d}_\Gamma}}.$
\begin{enumerate}[label=\roman*)]
\item One has
$$
 \int_{\Om\cap A_j^{r}}|\nabla \varphi_j|^2 \mathrm{d}_M \leqslant \frac{1}{r^2}\vol(\Om\cap A_j^{r}).\\
$$
The $A_j^{r}$'s are pairwise disjoint then $\sum_{j=1}^{4k}\vol(\Om\cap A_j^{r})\leqslant \vol(\Om)$. We deduce that at least $2k$ of $A_j^{r}$'s satisfy
\begin{equation}\label{eqh1}
\vol(\Om\cap A_j^{r})\leqslant\frac{\vol(\Om)}{k}.
\end{equation} 
Up to re-ordering, we assume that for the first $2k$ of the $A_j^{r}$'s we have \eqref{eqh1}.
Hence, 
$$ \int_{\Om\cap A_j^{r}}|\nabla \varphi_j|^2 \leqslant\frac{1}{r^2} \frac{\vol(\Om)}{k},\quad \forall~ 1\leqslant j\leqslant 2k.$$
\item \label{biii}
By the same arguments, at least $k$ of the $A_j^{r}$'s satisfy
\begin{equation}\label{alka2}
\vol(\Gamma\cap A_j^{r})\leqslant\frac{\vol(\Gamma)}{k}.
\end{equation} 
Up to re-ordering, we assume that for the first $k$ of the $A_j^{r}$'s  \eqref{alka2} holds.
Hence, 
$$ \int_{\Gamma\cap A_j^{r}}|\nabla f_i|^2\leqslant \frac{1}{r^2} \frac{\vol(\Gamma)}{k},\quad \forall~ 1\leqslant j\leqslant k.$$
\end{enumerate}
Since $\int_{\Gamma\cap A_j}{{\varphi_j}^2 \mathrm{d}_\Gamma}\geqslant \int_{\Gamma_0\cap A_j}{~\mathrm{d}_\Gamma}=\mu_0(A_j)\geqslant\frac{\vol(\Gamma_0)}{8N_r k}$, we have
\begin{align*}
R_\beta(\varphi_j)&\leqslant  \frac{8N_r k}{\vol(\Gamma_0)}\left[ \frac{1}{r^2} \frac{\vol(\Om)}{k}+\beta \frac{1}{r^2} \frac{\vol(\Gamma)}{k}   \right]\\
&= \frac{8N_r}{r^2\vol(\Gamma_0)} \left[  \vol(\Om)+\beta \vol(\Gamma)\right].
\end{align*}
However,
\begin{equation*}
\frac{1}{r^2}=\left(\frac{2^6}{r_{(4k)}}\right)^2=2^{12}\left(\frac{2}{\vol(\Omega)}\right)^{\frac{2}{n}} \left( \frac{4kC\rho_{n-1}}{I_0(\Omega)}\right)^{\frac{2}{n-1}}
\end{equation*}
and $$\vol(\Gamma_0)>   \frac{1}{16^{\frac{n-1}{n}}}I_0(\Omega)\vol(\Om)^{\frac{n-1}{n}}.$$
Thus, $$\frac{1}{r^2\vol(\Gamma_0)}\leqslant 2^{19}\frac{(kC\rho_{n-1})^{\frac{2}{n-1}}}{\vol(\Omega)^{1+\frac{1}{n}}I_0(\Omega)^{1+\frac{2}{n-1}}}.$$
We get
\begin{align*}
&R_\beta(\varphi_j)\leqslant
2^{21}N_r \frac{(kC\rho_{n-1})^{\frac{2}{n-1}}}{ \vol(\Om)^{1+\frac{1}{n}}I_0(\Omega)^{1+\frac{2}{n-1}}} \left[  \vol(\Om) +\beta \vol(\Gamma) \right]\\
&\leqslant  2^{21}N_r(C\rho_{n-1})^{\frac{2}{n-1}}
\left(\frac{k}{\vol(\Gamma)} \right)^{\frac{2}{n-1}}
 \frac{\vol(\Gamma)^{1+\frac{2}{n-1}}}{ \vol(\Om)^{1+\frac{1}{n}}I_0(\Omega)^{1+\frac{2}{n-1}}} \left[  \frac{\vol(\Om)}{\vol(\Gamma)} +\beta  \right]\\
&= 2^{21} N_r(C\rho_{n-1})^\frac{2}{n-1} \left(\frac{k}{\vol(\Gamma)}\right)^\frac{2}{n-1} \left(\frac{I(\Omega)}{I_0(\Omega)} \right)^{1+\frac{2}{n-1}} \left[ \frac{\vol(\Gamma)^\frac{1}{n-1}}{I(\Omega)^{1+\frac{1}{n-1}}}+\beta\right].
\end{align*}
Hence, 
\begin{align}
&\lambda_{W,k}^{\beta}(\Om)
\leqslant \max_{1\leqslant j\leqslant k} R_\beta(\varphi_j) \nonumber \\
&\leqslant 2^{21} N_r(C\rho_{n-1})^\frac{2}{n-1} \left(\frac{k}{\vol(\Gamma)}\right)^\frac{2}{n-1} \left(\frac{I(\Omega)}{I_0(\Omega)} \right)^{1+\frac{2}{n-1}} \left[ \frac{\vol(\Gamma)^\frac{1}{n-1}}{I(\Omega)^{1+\frac{1}{n-1}}}+\beta\right].\label{esti1}
%&\leqslant 2^{21} N_r(C\rho_{n-1})^\frac{2}{n-1} \left(\frac{k}{\vol(\Gamma)}\right)^\frac{2}{n-1} \left(\frac{I(\Omega)}{I_0(\Omega)} \right)^{1+\frac{2}{n-1}} \left[ \frac{\vol(\Gamma)^\frac{1}{n-1}}{I(\Omega)^{1+\frac{1}{n-1}}}+\beta\right].
\end{align}
\paragraph*{Case $\mu(B(x_{K},r))>\frac{\mu(X)}{4KN_r^2}$.}
From \eqref{eqg} one has
$\mu(B(x_j,r))\geqslant\frac{\mu(X)}{4KN_r^2}$
for every $1\leqslant j\leqslant K$.
 We consider, for $1\leqslant j \leqslant 4k$,the function $f_j$ supported
in $B_j:=B(x_j,2r)$ and defined by 
$$f_j(x):=
\begin{cases}
\min \{1,2-\frac{d(x_j,x)}{r}\}\quad &\forall~ x\in B_j,\\
0\quad &\forall~ x\in M\backslash B_j.
\end{cases}
$$
Set $A_j:=B(x_j,r)$, then the Rayleigh quotient of $f_j$ satisfies
$$R_\beta(f_j)\leqslant\frac{\int_{\Om\cap B_j}|\nabla f_j|^2 \mathrm{d}_M+\beta\int_{\Gamma\cap B_j}{|\nabla {f_j}|^2 \mathrm{d}_\Gamma}}{\int_{\Gamma\cap A_j}{{f_i}^2 \mathrm{d}_\Gamma}}.$$
\begin{enumerate}[label=\roman*)]
\item \label{ai} Since  for every $x\in A_j$, $f_j(x)=1$, one has
$$\int_{\Gamma\cap A_j}{{f_j}^2 \mathrm{d}_\Gamma}\geqslant\int_{\Gamma\cap A_j}\mathrm{d}_\Gamma\geqslant \mu(A_j)\geqslant \frac{\vol(\Gamma)}{16N_r^2 k}. $$
\item\label{aii} Set for $x\in M$, $d_j(x):=\dist(x_j, x)$, then
$$|\nabla f_j|\leqslant |\nabla(2-\frac{d_j(x)}{r})|=|\frac{1}{r}\nabla(d_j(x))|\leqslant\frac{1}{r}.$$ 
 By H\"older's inequality, we have 
 \begin{align*}
 \int_{\Om\cap B_j}|\nabla f_j|^2 \mathrm{d}_M &\leqslant \left(\int_{\Om\cap B_j}|\nabla f_j|^n \mathrm{d}_M \right)^{\frac{2}{n}}\left(\int_{\Om\cap B_j}\mathrm{d}_M \right)^{1-\frac{2}{n}}\\
 &\leqslant \left(\frac{1}{r^n}\vol(B_j) \right)^{\frac{2}{n}}\left( \vol(\Om\cap B_j) \right)^{1-\frac{2}{n}}.\\
 \end{align*}
Notice that $B_j\cap\Gamma\supset A_j\cap\Gamma \neq \emptyset$. Let $x_j'\in B_j\cap\Gamma$, one has $B_j\subset B(x_j',4r)$. Since $4r\leqslant r_K<r_-(\Gamma)$, 
 $$\vol(B_j)\leqslant \vol(B(x_j',4r))< C\omega_n(4r)^n.$$
In addition, since the $B_j$'s are pairwise disjoint, we have 
$$\sum_{j=1}^{4k}\vol(\Om\cap B_j)\leqslant \vol(\Om).$$ 
We deduce that at least $2k$ of $B_j$'s satisfy
\begin{equation}
\vol(\Om\cap B_j)\leqslant\frac{\vol(\Om)}{k}.\label{alk1}
\end{equation} 
Up to re-ordering, we assume that for the first $2k$ of the $B_j$'s \eqref{alk1} holds.
Hence, 
$$ \int_{\Om\cap B_j}|\nabla f_j|^2 \leqslant (C\omega_n4^n)^{\frac{2}{n}}\left(\frac{\vol(\Om)}{k} \right)^{1-\frac{2}{n}},\quad \forall~ 1\leqslant j\leqslant 2k.$$
\item \label{aiii}
 Again the $B_j$'s are pairwise disjoint then $\sum_{j=1}^{4k}\vol(\Gamma\cap B_j)\leqslant \vol(\Gamma)$. Hence at least $k$ of the $B_j$'s satisfy
\begin{equation}\label{alk2}
\vol(\Gamma\cap B_j)\leqslant\frac{\vol(\Gamma)}{k}.
\end{equation} 
Up to re-ordering, we assume that for the first $k$ of the $B_j$'s, inequality holds \eqref{alk2}.
Thus, 
$$ \int_{\Gamma\cap B_j}|\nabla f_i|^2\leqslant \frac{1}{r^2}\left(\frac{\vol(\Gamma)}{k} \right)^{1-\frac{2}{n-1}},\quad \forall~ 1\leqslant j\leqslant k.$$
\end{enumerate}
Hence, one has
\begin{align*}
R_\beta(f_j) &\leqslant
 \frac{16N_r^2 k}{\vol(\Gamma)}\left[(C\omega_n4^n)^{\frac{2}{n}}\left(\frac{\vol(\Om)}{k} \right)^{1-\frac{2}{n}}+\beta \frac{1}{r^2} \frac{\vol(\Gamma)}{k} \right]\\
&\leqslant
2^8 N_r^2(C\omega_n)^{\frac{2}{n}} \left(\frac{k}{\vol(\Gamma)}\right)^{\frac{2}{n}} \left(\frac{\vol(\Omega)}{\vol(\Gamma)}\right)^{1-\frac{2}{n}} \\
&~+ \beta 2^{10} N_r^2(C\rho_{n-1})^{\frac{2}{n-1}}\left( \frac{k}{\vol(\Gamma)}\right)^{\frac{2}{n-1}} \left(\frac{I(\Omega)}{I_0(\Omega)} \right)^{\frac{2}{n-1}}.
\end{align*}
Since $\frac{I(\Omega)}{I_0(\Omega)}\geqslant 1$, regarding the right hand side of \eqref{esti1}, we have 
\begin{align*}
R_\beta(f_j)\leqslant &
2^8 N_r^2(C\omega_n)^{\frac{2}{n}} \left(\frac{k}{\vol(\Gamma)}\right)^{\frac{2}{n}} \left(\frac{\vol(\Omega)}{\vol(\Gamma)}\right)^{1-\frac{2}{n}} \\
&+\beta 2^{21} N_r^2(C\rho_{n-1})^{\frac{2}{n-1}}\left( \frac{k}{\vol(\Gamma)}\right)^{\frac{2}{n-1}} \left(\frac{I(\Omega)}{I_0(\Omega)} \right)^{1+\frac{2}{n-1}}.
\end{align*}
Then, in this case 
\begin{align}
\lambda_{W,k}^{\beta}(\Om)
\leqslant & \max_{1\leqslant j\leqslant k} R_\beta(\varphi_j) \nonumber \\
\leqslant &
2^8 N_r^2(C\omega_n)^{\frac{2}{n}}\left(\frac{k}{\vol(\Gamma)}\right)^{\frac{2}{n}} \left(\frac{\vol(\Omega)}{\vol(\Gamma)}\right)^{1-\frac{2}{n}} \nonumber \\
&+\beta 2^{21} N_r^2(C\rho_{n-1})^{\frac{2}{n-1}}\left( \frac{k}{\vol(\Gamma)}\right)^{\frac{2}{n-1}} \left(\frac{I(\Omega)}{I_0(\Omega)} \right)^{1+\frac{2}{n-1}}.\label{esti2}
\end{align}
From \eqref{esti1} and \eqref{esti2}, in both possible cases we have
\begin{align}
\lambda_{W,k}^{\beta}(\Om)
\leqslant &
 C_1 \left(\frac{k}{\vol(\Gamma)}\right)^{\frac{2}{n}} \left(\frac{\vol(\Omega)}{\vol(\Gamma)}\right)^{1-\frac{2}{n}}\nonumber \\
&+ C_2 \left(\frac{k}{\vol(\Gamma)}\right)^\frac{2}{n-1} \left(\frac{I(\Omega)}{I_0(\Omega)} \right)^{1+\frac{2}{n-1}} \left[\frac{\vol(\Omega)}{\vol(\Gamma)}+\beta\right],
\end{align}
where $C_1:= 2^8 N_r^2(C\omega_n)^{\frac{2}{n}}$ and $C_2:=2^{21} N_r(C\rho_{n-1})^\frac{2}{n-1}$. This ends the proof.
\end{proof}
When $n\geqslant 3$, Theorem \ref{mainanniv} can be extended to cover all eigenvalues as follows:
\begin{thm}\label{t1501}
Let $(M,g)$ be a complete Riemannian manifold of dimension $n\geqslant 3$ and let 
$\Omega\subset M$ be a bounded domain whose boundary $\Gamma$ is a
smooth hypersurface satisfying the hypothesis $(H_0)$. 
We assume that $M$, with respect to the distance associated to the metric $g$ satisfies the $(N,4)$-covering property for some discrete positive function $N$. \\
 Then, for every $k\geqslant 1$, one has
\begin{multline}
\lambda_{W,k}^{\beta}(\Om)
\leqslant C(\Omega,\beta)
+ C_1 \left(\frac{\vol(\Omega)}{\vol(\Gamma)}\right)^{1-\frac{2}{n}} \left(\frac{k}{\vol(\Gamma)}\right)^{\frac{2}{n}}\\
+ C_2  \left(\frac{I(\Omega)}{I_0(\Omega)} \right)^{1+\frac{2}{n-1}} \left[ \frac{\vol(\Omega)}{\vol(\Gamma)}+\beta\right]\left(\frac{k}{\vol(\Gamma)}\right)^\frac{2}{n-1},
\end{multline}
where the constants $C_1$ and $C_2$ are the same as in Theorem \ref{mainanniv} and
\begin{multline}
C(\Omega,\beta):=\frac{C_1}{\left(C\rho_{n-1}r_-(\Gamma)^{n-1}\right )^{\frac{2}{n}}} \left(\frac{\vol(\Omega)}{\vol(\Gamma)}\right)^{1-\frac{2}{n}}\\ + \frac{C_2}{\left(C\rho_{n-1}r_-(\Gamma)^{n-1}\right )^{\frac{2}{n-1}}}\left(\frac{I(\Omega)}{I_0(\Omega)} \right)^{1+\frac{2}{n-1}} \left[ \frac{\vol(\Omega)}{\vol(\Gamma)}+\beta\right].
\end{multline}
\end{thm}
\begin{proof}
For $1\leqslant k < K_0$, one has 
\begin{align}
\lambda_{W,k}^{\beta}(\Om)
\leqslant &\lambda_{W,K_0}^{\beta}(\Om)\nonumber\\
\leqslant &C_1 \left(\frac{K_0}{\vol(\Gamma)}\right)^{\frac{2}{n}} \left(\frac{\vol(\Omega)}{\vol(\Gamma)}\right)^{1-\frac{2}{n}}\nonumber \\
&+ C_2 \left(\frac{K_0}{\vol(\Gamma)}\right)^\frac{2}{n-1} \left(\frac{I(\Omega)}{I_0(\Omega)} \right)^{1+\frac{2}{n-1}} \left[ \frac{\vol(\Gamma)^\frac{1}{n-1}}{I(\Omega)^{1+\frac{1}{n-1}}}+\beta\right].\label{ineq18052020}
\end{align}
However, $K_0\leqslant  {\frac{I_0(\Omega)}{C\rho_{n-1}r_-(\Gamma)^{n-1}} \left(\frac{\vol(\Omega)}{2}\right)^{\frac{n-1}{n}}}+1 $, using the triangle inequality, we obviously have
for every $p\in\N_{\geqslant 2}$
\begin{align*}
K_0^{\frac{2}{p}}&\leqslant \left({\frac{I_0(\Omega)}{C\rho_{n-1}r_-(\Gamma)^{n-1}} \left(\frac{\vol(\Omega)}{2}\right)^{\frac{n-1}{n}}}\right)^{\frac{2}{p}}+1\\
&\leqslant \left(\frac{\vol(\Gamma)}{2^{\frac{n-1}{n}}C\rho_{n-1}r_-(\Gamma)^{n-1}}\frac{I_0(\Omega)}{I(\Omega)}\right )^{\frac{2}{p}}+1\\
&\leqslant \left(\frac{\vol(\Gamma)}{C\rho_{n-1}r_-(\Gamma)^{n-1}}\right )^{\frac{2}{p}}+k^{\frac{2}{p}}.
\end{align*}
We set $C_3:=\left(\frac{1}{C\rho_{n-1}r_-(\Gamma)^{n-1}}\right )^{\frac{2}{p}}$, replacing in \eqref{ineq18052020}, we get
\begin{align}
\lambda_{W,k}^{\beta}(\Om)
\leqslant &C_1 \left\{C_3+\left(\frac{k}{\vol(\Gamma)}\right)^{\frac{2}{n}}\right\} \left(\frac{\vol(\Omega)}{\vol(\Gamma)}\right)^{1-\frac{2}{n}}\nonumber \\
&+ C_2 \left\{C_3+\left(\frac{k}{\vol(\Gamma)}\right)^\frac{2}{n-1}\right\} \left(\frac{I(\Omega)}{I_0(\Omega)} \right)^{1+\frac{2}{n-1}} \left[ \frac{\vol(\Omega)}{\vol(\Gamma)}+\beta\right].
\end{align}
Rearranging terms in above inequality, we have
\begin{multline}
\lambda_{W,k}^{\beta}(\Om)
\leqslant C(\Omega,\beta)
+ C_1 \left(\frac{\vol(\Omega)}{\vol(\Gamma)}\right)^{1-\frac{2}{n}} \left(\frac{k}{\vol(\Gamma)}\right)^{\frac{2}{n}}\\
+ C_2  \left(\frac{I(\Omega)}{I_0(\Omega)} \right)^{1+\frac{2}{n-1}} \left[ \frac{\vol(\Omega)}{\vol(\Gamma)}+\beta\right]\left(\frac{k}{\vol(\Gamma)}\right)^\frac{2}{n-1},
\end{multline}
where 
\begin{multline}
C(\Omega,\beta):=\frac{C_1}{\left(C\rho_{n-1}r_-(\Gamma)^{n-1}\right )^{\frac{2}{n}}} \left(\frac{\vol(\Omega)}{\vol(\Gamma)}\right)^{1-\frac{2}{n}}\\ + \frac{C_2}{\left(C\rho_{n-1}r_-(\Gamma)^{n-1}\right )^{\frac{2}{n-1}}}\left(\frac{I(\Omega)}{I_0(\Omega)} \right)^{1+\frac{2}{n-1}} \left[ \frac{\vol(\Omega)}{\vol(\Gamma)}+\beta\right].
\end{multline}
The result follows applying Theorem \ref{mainanniv} when $k\geqslant K_0$.
\end{proof}

%\section{Proof of Theorem }
%\begin{thm}\label{eucl}
%Let $n\geqslant 2$ and $\Omega$ be a bounded domain with smooth boundary $\Gamma$,  in $\R^n$. Then, for every $k\leqslant 1$, we have  
% \begin{align*}
%\lambda_{W,k}^{\beta}(\Om)\leqslant&
% \zeta_1(n) k^{\frac{2}{n}} \frac{I(\Omega)^{1+\frac{1}{n-1}}}{\vol(\Gamma)^\frac{1}{n-1}}+ \nonumber \\
%& \zeta_2(n) \left(\frac{k}{\vol(\Gamma)}\right)^\frac{2}{n-1} {I(\Omega)}^{1+\frac{2}{n-1}} \left[ \frac{\vol(\Gamma)^\frac{1}{n-1}}{I(\Omega)^{1+\frac{1}{n-1}}}+\beta\right],
%\end{align*}
%where $\zeta_1(n):= 2^{6+10n+\frac{2}{n}}\omega_n^{\frac{2}{n}}$ and
% $\zeta_2(n):=\frac{ 2^{3n+23+^\frac{2}{n-1}} \rho_{n-1}^\frac{2}{n-1}}{\left(n\omega_n^\frac{1}{n} \right)^{1+\frac{2}{n-1}}}
%$.
%\end{thm}
\section{Proof of main theorems}\label{sec1701}
\begin{proof}[Proof of Theorem \ref{eucl}]
We have in the Euclidean case:
$$r_-(\Gamma)=+\infty,\quad C=2, \quad I_0(\Omega)=I_0(\R^n)=n\omega_n^\frac{1}{n},\quad N_r\leqslant 32^n, \forall~ r>0.$$
 Applying Theorem \ref{t1501}, we get for every $k\geqslant 1$
\begin{multline}
\lambda_{W,k}^{\beta}(\Om)
\leqslant  \zeta_1(n)  \left(\frac{\vol(\Omega)}{\vol(\Gamma)}\right)^{1-\frac{2}{n}} \left(\frac{k}{\vol(\Gamma)}\right)^{\frac{2}{n}}\\
+ \zeta_2(n) \left(\frac{I(\Omega)}{I_0(\Omega)} \right)^{1+\frac{2}{n-1}} \left[ \frac{\vol(\Omega)}{\vol(\Gamma)}+\beta\right]\left(\frac{k}{\vol(\Gamma)}\right)^\frac{2}{n-1},
\end{multline}
where $\zeta_1(n):= 2^{10(n+1)}\omega_n^{\frac{2}{n}}$ and
 $\zeta_2(n):= 2^{10(n+3)}\rho_{n-1}^\frac{2}{n-1}
$. The result follows replacing $I_0(\Omega)$ by $n\omega^{\frac{1}{n}}$.
\end{proof}
\begin{proof}[Proof of Corollary \ref{eucl20052020}]
From Theorem \ref{eucl}, one has 
\begin{align}
\lambda_{W,k}^{\beta}(\Om)
&\leqslant
\zeta_1(n)  \left(\frac{\vol(\Omega)}{\vol(\Gamma)}\right)^{1-\frac{2}{n}} \left(\frac{k}{\vol(\Gamma)}\right)^{\frac{2}{n}}\nonumber\\
&~ + \zeta_2(n) I(\Omega)^{1+\frac{2}{n-1}} \left[ \frac{\vol(\Omega)}{\vol(\Gamma)}+\beta\right]\left(\frac{k}{\vol(\Gamma)}\right)^\frac{2}{n-1}\label{ineq20052020}\\
&=
 \Bigg\{
 \zeta_1(n)  \left(\frac{\vol(\Omega)}{\vol(\Gamma)}\right)^{1-\frac{2}{n}} \left(\frac{\vol(\Gamma)}{k}\right)^{\frac{2}{n(n-1)}}\nonumber\\
&~+ \zeta_2(n) I(\Omega)^{1+\frac{2}{n-1}} \left[ \frac{\vol(\Omega)}{\vol(\Gamma)}+\beta\right]
\Bigg\}
\left(\frac{k}{\vol(\Gamma)}\right)^{\frac{2}{n-1}}.\nonumber
\end{align}
\begin{enumerate}
\item If $\zeta_1(n)  \left(\frac{\vol(\Omega)}{\vol(\Gamma)}\right)^{1-\frac{2}{n}} \left(\frac{\vol(\Gamma)}{k}\right)^{\frac{2}{n(n-1)}}<1$, then 
$$\lambda_{W,k}^{\beta}(\Om)
< C_2(\Omega, \beta)\left(\frac{k}{\vol(\Gamma)}\right)^{\frac{2}{n-1}}.$$
\item Otherwise, $\zeta_1(n)  \left(\frac{\vol(\Omega)}{\vol(\Gamma)}\right)^{1-\frac{2}{n}} \left(\frac{\vol(\Gamma)}{k}\right)^{\frac{2}{n(n-1)}}\geqslant 1$. That is,
$$\frac{k}{\vol(\Gamma)}\leqslant \left[\zeta_1(n)  \left(\frac{\vol(\Omega)}{\vol(\Gamma)}\right)^{\frac{n-2}{n}} \right]^{\frac{n(n-1)}{2}},$$
\begin{equation}
\begin{cases}
\left(\frac{k}{\vol(\Gamma)}\right)^{\frac{2}{n}}\leqslant \zeta_1^{n-1}(n)  \left(\frac{\vol(\Omega)}{\vol(\Gamma)}\right)^{\frac{(n-2)(n-1)}{n}}\\
\left(\frac{k}{\vol(\Gamma)}\right)^{\frac{2}{n-1}}\leqslant \zeta_1^n(n)  \left(\frac{\vol(\Omega)}{\vol(\Gamma)}\right)^{n-2}.
\end{cases}
\end{equation}
Replacing in \eqref{ineq20052020}, we  get 
\begin{align*}
\lambda_{W,k}^{\beta}(\Om)
\leqslant & \zeta_1^n(n)  \left(\frac{\vol(\Omega)}{\vol(\Gamma)}\right)^{n-2}\\
&+ \zeta_1^n(n)  \left(\frac{\vol(\Omega)}{\vol(\Gamma)}\right)^{n-2}\zeta_2(n) I(\Omega)^{1+\frac{2}{n-1}} \left[ \frac{\vol(\Omega)}{\vol(\Gamma)}+\beta\right]\\
&\leqslant \zeta_1^n(n)  \left(\frac{\vol(\Omega)}{\vol(\Gamma)}\right)^{n-2}\left\{1+\zeta_2(n) I(\Omega)^{1+\frac{2}{n-1}} \left[ \frac{\vol(\Omega)}{\vol(\Gamma)}+\beta\right]\right\}\\
&= C_1(\Omega, \beta)
\end{align*}
\end{enumerate}
 In  both cases, one has $\lambda_{W,k}^{\beta}(\Om)
\leqslant C_1(\Omega, \beta)+C_2(\Omega, \beta)\left(\frac{k}{\vol(\Gamma)}\right)^{\frac{2}{n-1}}$.
\end{proof}
\begin{proof}[Proof of Theorem \ref{t16012}]
Let $r>0$, we denote by $\nu(n,-\kappa^2,r)$ (respectively $\nu_\partial(n,-\kappa^2,r)$)  the volume of a ball (respectively a sphere) of radius $r$ in the constant curvature model space $M_{-\kappa^2}^n$.
As a consequence of the relative Bishop-Gromov volume comparison theorem, we have the following volume and area comparisons, for every $r>0$ and $x\in M$:
\begin{equation*}
 \vol (B(x,r))\leqslant \nu(n,-\kappa^2, r)\quad \text{and}\quad \vol (\partial B(x,r))\leqslant \nu_{\partial}(n,-\kappa^2, r).
\end{equation*}
The sphere of radius $r$ in the model space $M_{-\kappa^2}^n$ has area 
$$\nu_\partial(n,-\kappa^2,r)=\rho_{n-1}sn_{-\kappa^2}(r)^{n-1}$$
and the ball of radius $r$ has volume 
$$\nu(n,-\kappa^2,r)=\rho_{n-1}\int_0^r sn_{-\kappa}(t)^{n-1}\mathrm{d} t,$$
where   $sn_\varkappa:\R\longrightarrow\R$ is defined by
\begin{equation*}
sn_\varkappa(t)=
\begin{cases}
\frac{1}{\sqrt{\varkappa}}\sin(\sqrt{\varkappa}t)\quad & \text{if } \varkappa>0\\
t \quad &  \text{if }\varkappa=0\\
\frac{1}{\sqrt{-\varkappa}}\sinh(\sqrt{-\varkappa}t)\quad &  \text{if }\varkappa<0.
\end{cases}
\end{equation*}
 
%\paragraph*{\underline{Case $\kappa>0$}}
\begin{align*}
\rho_{n-1}\int_0^{r}sn_{-\kappa^2}(r)^{n-1}\mathrm{d}t&=\rho_{n-1}\int_0^{r}\left((\frac{1}{\kappa}\sinh(\kappa t)\right)^{n-1}\mathrm{d}t \\
&\leqslant \rho_{n-1} \int_0^{r} \left[te^{\kappa t}\right]^{n-1}\mathrm{d}t\\
&\leqslant \rho_{n-1}e^{r(n-1)\kappa} \int_0^{r} t^{n-1}\mathrm{d}t\\
&\leqslant \omega_nr^ne^{r(n-1)\kappa} 
\end{align*}
and
$$\rho_{n-1} sn_{-\kappa^2}(r)^{n-1}\leqslant e^{r(n-1)\kappa} \rho_{n-1}r^{n-1}.$$
Hence, for every $0<r<1$ and $x\in M$, we have
\begin{equation}
  \vol(B(x,r))<C \omega_nr^n\quad\text{and}\quad \vol(\partial B(x,r))<C \rho_{n-1}r^{n-1}, 
\end{equation}
with $C:= e^{n\kappa}$.

On the other hand, for every $0<r<1$ and $x\in M$, $B(x,r)$ can be covered by $N:=2^{5n}e^{4r(n-1)\kappa t}<2^{5n}e^{4(n-1)\kappa}$ balls of radius $\frac{r}{4}$.
Indeed, take
 $\{B(x_i,\frac{r}{8})\}_{i=1}^N$ a maximal family of disjoint balls with center $x_i\in B(x,r)$. By the maximality assumption, the family $\{B(x_i,\frac{r}{4})\}_{i=1}^N$ covers $B(x,r)$. Let $i_0\in\{1,\ldots,N\}$ such that
 $$\vol \left(B(x_{i_0},\frac{r}{8})\right)=\min_{1\leqslant i\leqslant N}\vol\left(B(x_i,\frac{r}{8})\right).$$
Then one has 
  $$N\vol(B(x_{i_0},\frac{r}{8}))\leqslant \sum_{1\leqslant i\leqslant N}\vol(B(x_{i},\frac{r}{8}))$$
since the balls  $B(x_{i},\frac{r}{8})$ are pairwise disjoint. In addition, $B(x_{i},\frac{r}{8})\subset B(x_{i},r+\frac{r}{8})$  for every $x_i\in B(x,r)$. Hence
$ N\vol(B(x_{i_0},\frac{r}{8}))\leqslant \vol(B(x,\frac{9r}{8}))$,
$$N\leqslant \frac{\vol(B(x,\frac{9r}{8}))}{\vol(B(x_{i_0},\frac{r}{8}))}
\leqslant \frac{\vol(B(x,2r))}{\vol(B(x_{i_0},\frac{r}{8}))}
\leqslant \frac{\vol(B(x_{i_0},4r))}{\vol(B(x_{i_0},\frac{r}{8}))}.$$
Using the Relative volume comparison theorem (Bishop 1964, Gromov 1980), one has 
$$\frac{\vol(B(x_{i_0},4r))}{\vol(B(x_{i_0},\frac{r}{8}))}\leqslant \frac{\nu(n,-\kappa,4r)}{\nu(n,-\kappa,\frac{r}{8})},$$
where $\nu(n,\kappa,r)$ denotes the volume of a ball of radius $r$ in the constant curvature model space $M_\kappa^n$. Then
\begin{align*}
N&\leqslant\frac{\int_0^{4r}\sinh^{n-1}(\kappa t)\mathrm{d}t}{\int_0^{\frac{r}{8}}\sinh^{n-1}(\kappa t)\mathrm{d}t}
\leqslant \frac{\int_0^{4r} \left[(\kappa t)e^{\kappa t}\right]^{n-1}\mathrm{d}t}{\int_0^{\frac{r}{8}}(\kappa t)^{n-1}\mathrm{d}t}\\
&\leqslant \frac{e^{4r(n-1)\kappa} \int_0^{4r} t^{n-1}\mathrm{d}t}{\int_0^{\frac{r}{8}}t^{n-1}\mathrm{d}t}=2^{5n}e^{4r(n-1)\kappa t}<2^{5n}e^{4(n-1)\kappa}.
\end{align*}
 Then applying Theorem \ref{t1501} with $$r_-(\Gamma)=1,\quad C=e^{n\kappa}, \quad N_r=2^{5n}e^{4(n-1)\kappa},$$
we get, for every $k\geqslant 1$,
\begin{multline*}
\lambda_{W,k}^{\beta}(\Om)
\leqslant e^{c_0(n)\kappa}
\Bigg\{
C(\Omega,\beta)\\
+ c_1(n) \left(\frac{\vol(\Omega)}{\vol(\Gamma)}\right)^{1-\frac{2}{n}} \left(\frac{k}{\vol(\Gamma)}\right)^{\frac{2}{n}}\\
+ c_2(n)  \left(\frac{I(\Omega)}{I_0(\Omega)} \right)^{1+\frac{2}{n-1}} \left[ \frac{\vol(\Omega)}{\vol(\Gamma)}+\beta\right]\left(\frac{k}{\vol(\Gamma)}\right)^\frac{2}{n-1}
\Bigg\},
\end{multline*}
where 
\begin{equation*}
C(\Omega,\beta):=c'_1(n) \left(\frac{\vol(\Omega)}{\vol(\Gamma)}\right)^{1-\frac{2}{n}}+ c'_2(n)\left(\frac{I(\Omega)}{I_0(\Omega)} \right)^{1+\frac{2}{n-1}} \left[ \frac{\vol(\Omega)}{\vol(\Gamma)}+\beta\right]
\end{equation*}
and
the constants $c_0(n)$, $c_1(n)$, $c_2(n)$, $c'_1(n)$ and $c'_2(n)$ depend only on $n$.

Following the same arguments as the proof of Corollary \ref{eucl20052020}, we have 
for every $k\geqslant 1$, one has
\begin{equation*}
\lambda_{W,k}^{\beta}(\Om)
\leqslant e^{c_0(n)\kappa}
\Bigg\{
\overline{C}_1(\Omega,\beta)+\overline{C}_2(\Omega,\beta)\left(\frac{k}{\vol(\Gamma)}\right)^\frac{2}{n-1}
\Bigg\},
\end{equation*}
where $\overline{C}_2(\Omega,\beta)=1+c_2(n)\left(\frac{I(\Omega)}{I_0(\Omega)} \right)^{1+\frac{2}{n-1}} \left[ \frac{\vol(\Omega)}{\vol(\Gamma)}+\beta\right]$ and\\
$\overline{C}_1(\Omega,\beta)=C(\Omega,\beta)+c_1^n(n)\overline{C}_2(\Omega,\beta).$

%%%%%%%%%%%%
 \begin{enumerate}[label={}]
    
    \item[-] If $\kappa\leqslant 1$, then 
    \begin{equation*}
\lambda_{W,k}^{\beta}(\Om)
\leqslant e^{c_0(n)}
\Bigg\{
\overline{C}_1(\Omega,\beta)+\overline{C}_2(\Omega,\beta)\left(\frac{k}{\vol(\Gamma)}\right)^\frac{2}{n-1}
\Bigg\},
\end{equation*}
which implies \eqref{ineq22052020}.
    \item[-] Otherwise, we assume that  $\ric (M,g) \geqslant -(n-1)\kappa^2 g$ with $\kappa> 1$. Then the Ricci curvature $\ric (M,\tilde{g})$ of the rescaled metric $\tilde{g}:=\kappa^2 g$ is bounded by $-(n-1)\tilde{g}$. We mark with a tilde quantities associated with the metric ${\tilde {g}}$, while those unmarked with such will be still associated with the metric $g$. Then we have
\begin{equation}
\lambda_{W,k}^{\beta}(\tilde{\Om})
\leqslant e^{c_0(n)}
\Bigg\{
\overline{C}_1(\tilde{\Omega},\beta)+\overline{C}_2(\tilde{\Omega},\beta)\left(\frac{k}{\vol(\tilde{\Gamma)}}\right)^\frac{2}{n-1}
\Bigg\}.\label{ineq21052020}
\end{equation}
However $\vol(\tilde{\Omega})=\mathrm{Vol}_{\tilde{g}}(\Om)=\kappa^{n}\vol(\Om)$ and $\vol(\tilde{\Gamma})=\kappa^{n-1}\vol(\Gamma)$.
 Thus,
\begin{align}
\overline{C}_2(\tilde{\Omega},\beta)
&=1+c_2(n)\left(\frac{I(\tilde{\Omega})}{I_0(\tilde{\Omega)}} \right)^{1+\frac{2}{n-1}} \left[ \frac{\vol(\tilde{\Omega)}}{\vol(\tilde{\Gamma)}}+\beta\right]\nonumber\\
&=1+c_2(n)\left(\frac{I(\Omega)}{I_0(\Omega)} \right)^{1+\frac{2}{n-1}} \left[ \kappa\frac{\vol(\Omega)}{\vol(\Gamma)}+\beta\right].\label{ineq21052020a}
\end{align}
Likewise, since
\begin{align*}
C(\tilde{\Omega},\beta)
&=c'_1(n) \left(\kappa\frac{\vol(\Omega)}{\vol(\Gamma)}\right)^{1-\frac{2}{n}}+ c'_2(n)\left(\frac{I(\Omega)}{I_0(\Omega)} \right)^{1+\frac{2}{n-1}} \left[\kappa \frac{\vol(\Omega)}{\vol(\Gamma)}+\beta\right],
\end{align*}
we have
\begin{align}
\overline{C}_1(\tilde{\Omega},\beta)&=C(\tilde{\Omega},\beta)+c_1^n(n)\overline{C}_2(\tilde{\Omega},\beta)\nonumber\\
&=c_1^n(n)+c'_1(n) \left(\kappa\frac{\vol(\Omega)}{\vol(\Gamma)}\right)^{1-\frac{2}{n}}\nonumber\\
&~+\Big(c_1^n(n)c_2(n)+c'_2(n)\Big)\left(\frac{I(\Omega)}{I_0(\Omega)} \right)^{1+\frac{2}{n-1}} \left[ \kappa\frac{\vol(\Omega)}{\vol(\Gamma)}+\beta\right]
.
\label{ineq21052020b}
\end{align}
We set $\overline{\zeta}(n):=\max\{1, c_2(n),  c_1^n(n), c'_1(n),c_1^n(n)c_2(n)+c'_2(n)\}$ so that
\begin{align*}
\overline{C}_2(\tilde{\Omega},\beta)&
\leqslant\overline{\zeta}(n)\left\{1+\left(\frac{I(\Omega)}{I_0(\Omega)} \right)^{1+\frac{2}{n-1}} \left[ \kappa\frac{\vol(\Omega)}{\vol(\Gamma)}+\beta\right]\right\}\text{ and }\\
\overline{C}_1(\tilde{\Omega},\beta)&
\leqslant \overline{\zeta}(n)\left\{ 1+ \left(\kappa\frac{\vol(\Omega)}{\vol(\Gamma)}\right)^{1-\frac{2}{n}}+\left(\frac{I(\Omega)}{I_0(\Omega)} \right)^{1+\frac{2}{n-1}} \left[ \kappa\frac{\vol(\Omega)}{\vol(\Gamma)}+\beta\right]\right\}
.
\end{align*}
In addition, since $\kappa>1$, for all  $u\in \mathfrak{V}_\beta\backslash\{0\}$ we have
 \begin{equation*}
\tilde{R}_\beta(u)=\frac{\kappa\int_\Om{|\nabla u|^2 \mathrm{d}_M+\beta\int_{\Gamma}{|\nabla_\Gamma u|^2 \mathrm{d}_\Gamma}}}{\kappa^2\int_{\Gamma}{u^2 \mathrm{d}_\Gamma}}\geqslant\frac{1}{\kappa^2} {R}_\beta(u).
\end{equation*}
Every orthonormal basis of  a $ k$-dimensional  subspaces $V\in\mathfrak{V}(k) $ of   $ \mathfrak{V}_\beta$ remains orthogonal with the metric $\tilde{g}$, then using the variation characterisation with \eqref{ineq21052020}, \eqref{ineq21052020a} and \eqref{ineq21052020b}, we have
\begin{align}
\lambda_{W,k}^{\beta}(\Om)&\leqslant\kappa^2\lambda_{W,k}^{\beta}(\tilde{\Om})
\leqslant \kappa^2e^{c_0(n)}
\Bigg\{
\overline{C}_1(\Omega,\beta)+\overline{C}_2(\tilde{\Omega},\beta)\left(\frac{k}{\vol(\tilde{\Gamma)}}\right)^\frac{2}{n-1}
\Bigg\}\nonumber\\
&\leqslant
\kappa^2  \zeta(n)\left\{ 1+ \left(\kappa\frac{\vol(\Omega)}{\vol(\Gamma)}\right)^{1-\frac{2}{n}}+\left(\frac{I(\Omega)}{I_0(\Omega)} \right)^{1+\frac{2}{n-1}} \left[ \kappa\frac{\vol(\Omega)}{\vol(\Gamma)}+\beta\right]\right\} \nonumber\\
 &~ +\zeta(n)\Bigg\{1+\left(\frac{I(\Omega)}{I_0(\Omega)} \right)^{1+\frac{2}{n-1}} \left[ \kappa\frac{\vol(\Omega)}{\vol(\Gamma)}+\beta\right]\Bigg\}\left(\frac{k}{\vol(\Gamma)}\right)^\frac{2}{n-1}
,
\end{align}
where $ \zeta(n)=e^{c_0(n)}\overline{\zeta}(n)$ is a dimensional constant.
    %To make the notation precise, a reference to the (a) coordinates should be added.
 \end{enumerate}
%%%%%%%%%%%%%%%%%%
\end{proof}

\begin{proof}[Proof of Theorem \ref{t16011}]
We have for every $r>0$, $sn_{0}(r)=r$, then for every $r>0$ and $x\in M$, one has
\begin{equation*}
 \vol (B(x,r))\leqslant \nu(n,0, r)=\rho_{n-1}\int_0^r  t^{n-1}\mathrm{d}t=\omega_nr^n
\end{equation*}
and
\begin{equation*}
\vol (\partial B(x,r))\leqslant \nu_{\partial}(n,0, r)=\rho_{n-1} (r)^{n-1},\quad \forall~ x\in M.
\end{equation*}
On the other hand, for every $r>0$ and $x\in M$, $B(x,r)$ can be covered by $N:=32^{n}$ balls of radius $\frac{r}{4}$.
As above, we take
 $\{B(x_i,\frac{r}{8})\}_{i=1}^N$ a maximal family of disjoint balls with center $x_i\in B(x,r)$. By the maximality assumption, the family $\{B(x_i,\frac{r}{4})\}_{i=1}^N$ covers $B(x,r)$. Let $i_0\in\{1,\ldots,N\}$ such that
 $$\vol (B(x_{i_0},\frac{r}{8}))=\min_{1\leqslant i\leqslant N}\vol(B(x_i,\frac{r}{8})).$$
  Then, since the balls  $B(x_{i},\frac{r}{8})$ are pairwise disjoint, one has 
  $$N\vol(B(x_{i_0},\frac{r}{8}))\leqslant \sum_{1\leqslant i\leqslant N}\vol(B(x_{i},\frac{r}{8})).$$
In addition, $B(x_{i},\frac{r}{8})\subset B(x_{i},r+\frac{r}{8})$  for every $x_i\in B(x,r)$, so
$$ N\vol(B(x_{i_0},\frac{r}{8}))\leqslant \vol(B(x,\frac{9r}{8})).$$ 
Hence,
$$N\leqslant \frac{\vol(B(x,\frac{9r}{8}))}{\vol(B(x_{i_0},\frac{r}{8}))}
\leqslant \frac{\vol(B(x,2r))}{\vol(B(x_{i_0},\frac{r}{8}))}
\leqslant \frac{\vol(B(x_{i_0},4r))}{\vol(B(x_{i_0},\frac{r}{8}))}.$$
Using the volume comparison theorem (Bishop 1964, Gromov 1980), one has 
$$\frac{\vol(B(x_{i_0},4r))}{\vol(B(x_{i_0},\frac{r}{8}))}\leqslant \frac{\omega_n(4r)^n}{\omega_n(\frac{r}{8})^n}\leqslant 32^n.$$
 Then follows from Theorem \ref{t1501} with $r_-(\Gamma)=+\infty$, $C=2$ and $N_r=32^n$.

\end{proof}

\bibliographystyle{plain}
%  \bibliography{Bibli}

\begin{thebibliography}{10}

\bibitem{CEG2013}
Bruno Colbois, Ahmad El~Soufi, and Alexandre Girouard.
\newblock Isoperimetric control of the spectrum of a compact hypersurface.
\newblock {\em J. Reine Angew. Math.}, 683:49--65, 2013.

\bibitem{CM}
Bruno Colbois and Daniel Maerten.
\newblock Eigenvalues estimate for the {N}eumann problem of a bounded domain.
\newblock {\em J. Geom. Anal.}, 18(4):1022--1032, 2008.

\bibitem{Croke1984}
Christopher~B. Croke.
\newblock A sharp four-dimensional isoperimetric inequality.
\newblock {\em Comment. Math. Helv.}, 59(2):187--192, 1984.

\bibitem{DKL}
M.~Dambrine, D.~Kateb, and J.~Lamboley.
\newblock An extremal eigenvalue problem for the {W}entzell-{L}aplace operator.
\newblock {\em Ann. Inst. H. Poincar\'{e} Anal. Non Lin\'{e}aire},
  33(2):409--450, 2016.

\bibitem{Feng2}
Feng Du, Qiaoling Wang, and Changyu Xia.
\newblock Estimates for eigenvalues of the {W}entzell-{L}aplace operator.
\newblock {\em J. Geom. Phys.}, 129:25--33, 2018.

\bibitem{Favini2002}
Angelo Favini, Gis{\`e}le~Ruiz Goldstein, Jerome~A. Goldstein, and Silvia
  Romanelli.
\newblock The heat equation with generalized wentzell boundary condition.
\newblock {\em Journal of Evolution Equations}, 2(1):1--19, Mar 2002.

\bibitem{Gal2015}
Ciprian~G. Gal.
\newblock The role of surface diffusion in dynamic boundary conditions: Where
  do we stand?
\newblock {\em Milan Journal of Mathematics}, 83(2):237--278, Dec 2015.

\bibitem{GhomiSpruck2019}
Mohammad Ghomi and Joel Spruck.
\newblock Total curvature and the isoperimetric inequality in cartan-hadamard
  manifolds, 2019.

\bibitem{GP}
Alexandre Girouard and Iosif Polterovich.
\newblock Spectral geometry of the {S}teklov problem (survey article).
\newblock {\em J. Spectr. Theory}, 7(2):321--359, 2017.

\bibitem{goldstein}
Gis\`ele~Ruiz Goldstein.
\newblock Derivation and physical interpretation of general boundary
  conditions.
\newblock {\em Adv. Differential Equations}, 11(4):457--480, 2006.

\bibitem{Kennedy2010}
J.~B. Kennedy.
\newblock On the isoperimetric problem for the higher eigenvalues of the
  {R}obin and {W}entzell {L}aplacians.
\newblock {\em Z. Angew. Math. Phys.}, 61(5):781--792, 2010.

\bibitem{Kleiner1992}
Bruce Kleiner.
\newblock An isoperimetric comparison theorem.
\newblock {\em Invent. Math.}, 108(1):37--47, 1992.

\bibitem{sandgren}
Lennart Sandgren.
\newblock A vibration problem.
\newblock {\em Medd. Lunds Univ. Mat. Sem.}, 13:1--84, 1955.

\bibitem{Taylor}
Michael~E. Taylor.
\newblock {\em Partial differential equations {I}. {B}asic theory}, volume 115
  of {\em Applied Mathematical Sciences}.
\newblock Springer, New York, second edition, 2011.

\bibitem{Wentzel}
A.~D. Ventcel'.
\newblock On boundary conditions for multi-dimensional diffusion processes.
\newblock {\em Theor. Probability Appl.}, 4:164--177, 1959.

\bibitem{Weil1926}
Andr{\'e} Weil.
\newblock Sur les surfaces a courbure negative.
\newblock {\em CR Acad. Sci. Paris}, 182(2):1069--71, 1926.

\bibitem{Feng1}
Changyu Xia and Qiaoling Wang.
\newblock Eigenvalues of the {W}entzell-{L}aplace operator and of the fourth
  order {S}teklov problems.
\newblock {\em J. Differential Equations}, 264(10):6486--6506, 2018.

\end{thebibliography}

% ------------------------------------------------------------------------
\end{document}